\documentclass{amsart}

\usepackage{graphicx}
\usepackage{amssymb}
\usepackage{latexsym}

\usepackage{amsmath,amsthm}

\usepackage[all]{xy}
\usepackage{mathrsfs}

\usepackage{amsfonts}
\usepackage{amsmath}

\usepackage{mathtools}
\usepackage{dsfont}
\usepackage{url}
\usepackage{enumerate}

\usepackage[hyperindex=true, colorlinks=false]{hyperref}

\newtheorem{theorem}{Theorem}

\newtheorem{lemma}{Lemma}[section]

\newtheorem{corollary}[lemma]{Corollary}
\newtheorem{proposition}[lemma]{Proposition}

\newtheorem{question}[lemma]{Question}

\theoremstyle{definition}

\newtheorem{example}[lemma]{Example}
\newtheorem{definition}[lemma]{Definition}
\newtheorem{definition-lemma}[lemma]{Definition-Lemma}
\newtheorem{definition-theorem}[lemma]{Definition-Theorem}

\newtheorem{remark}[lemma]{Remark}

\newcommand{\eps}{{\varepsilon}}

\newcommand{\R}{\mathbb{R}}

\newcommand{\Om}{\Omega}
\newcommand{\la}{\lambda}

\newcommand{\pa}{\partial}

\newcommand{\fT}{\mathcal{T}}

\newcommand{\begla}{\begin{equation}}
\newcommand{\beglab}[1]{\begin{equation}	\label{#1}}
\newcommand{\edla}{\end{equation}}

\newcommand{\defeq}{\coloneqq}
\newcommand{\col}{\colon\thinspace}

\newcommand{\ie}{{\it i.e.}\ }

\begin{document}


\title{Simple Smale flows and their templates on $S^3$}

\author{Xiang Liu}
\address{School of Mathematical Sciences, Capital Normal University, Beijing 100048, P. R. China}
\email{nz\_liu1989@163.com}

\author{Xuezhi Zhao}
\address{School of Mathematical Sciences, Capital Normal University, Beijing 100048, P. R. China}
\email{zhaoxve@mail.cnu.edu.cn}

\thanks{Supported by NSF of China (No. 11961131004)}



\begin{abstract}
The embedded template is a geometric tool in dynamics being used to model knots and links as periodic orbits of $3$-dimensional flows. We prove that for an embedded template in $S^3$ with fixed homeomorphism type, its boundary as a trivalent spatial graph is a complete isotopic invariant. Moreover, we construct an invariant of embedded templates by Kauffman's invariant of spatial graphs, which is a set of knots and links. As application, the isotopic classification of simple Smale flows on $S^3$ is discussed.
\end{abstract}

\maketitle



\section{Introduction}

To study the knots and links occurring in flows on $S^3$, Birman and Williams \cite{BW1, BW2} introduced the template (they called it the knot holder) as a geometric model. Roughly speaking, a template is a compact branched surface with an associated semi-flow on it, whose periodic orbits are in bijective correspondence with the periodic orbits of the original $3$-dimensional flow.

As in knot theory, the embedded template is studied by the template diagram with planar moves \cite{KSS}. Two template diagrams correspond to isotopic templates if and only if one can be obtained from the other by a finite sequence of plane isotopies and planar moves. However, from the point of view of dynamics, two embedded templates are identified if they are isotopic in ambient space with additional switch move and splitting move \cite{BW2}. These two template moves may change the homeomorphism type of a template, thus bring an obstruction to define topological invariant of embedded templates.

In this paper, we fix the homeomorphism type of an embedded template. 
Then the boundary of the embedded template is a trivalent spatial graph, which is in fact a complete invariant of the original template with specifical homeomorphism type (Theorem~\ref{crit}). Furthermore, we use Kauffman's invariant of spatial graphs \cite{K} to make it easy for comparing template boundaries. It leads an isotopic invariant which is a set of knots and links (Theorem~\ref{taut}). Another advantage of this link-set invariant is the veracity of describing the template with twisted and knotted bands. Knots as invariants of dynamics are not unusual such as in \cite{BG}.

The isotopy of flows is also under consideration. Particularly for simple Smale flows, we construct a spatial graph invariant and a generalization of Kauffman's invariant involving the unique attractor (Theorem~\ref{lemg}). They can classify the simple Smale flows on $S^3$ with Lorenz-like templates \cite{Su, Y} and the universal template with its analogues \cite{HS, AS}, and the former is a complete isotopic invariant for general case (Theorem~\ref{flowg}) while the completeness of the later is unknown.

This paper is organized as follows. In Section 2, we give a brief account of template theory, and fix some notions. The boundary spatial graph of a template is discussed in Section 3. Our main result lies in Section 4. We show that the isotopy class of a template is totally determined by that of its boundary spatial graph. Section 5 devotes an application of our results into simple Smale flow on $S^3$.


\section{Templates and ambient isotopy}   \label{secBCHD}

In this section, we concentrate on $3$-dimensional manifolds. We state the template theory briefly and discuss ambient isotopy acting on embedded templates.


\begin{definition} \cite{BW2, GHS}
A template $(\fT, \varphi_t)$ is a compact branched $2$-manifold $\fT$ constructed from two types of charts, called joining charts and splitting charts, together with a semi-flow $\varphi_t\col \fT \to \fT$. The gluing maps between charts must respect the semi-flow and act linearly on the edges.
\end{definition}


Here a semi-flow is the same as a flow except that one cannot go backward in time uniquely, and in a template one cannot back up uniquely at a branch line. The semi-flows are usually indicated by arrows on charts.

\begin{example}
Lorenz template $\mathcal{L}(0,0)$, horseshoe template $\mathcal{L}(0,1)$, and template $\mathcal{L}(1,1)$ as bellow. Any two of them are not homeomorphic as branched surfaces, since the numbers of non-orientable bands of them are $0,1$ and $2$, respectively.
\end{example}

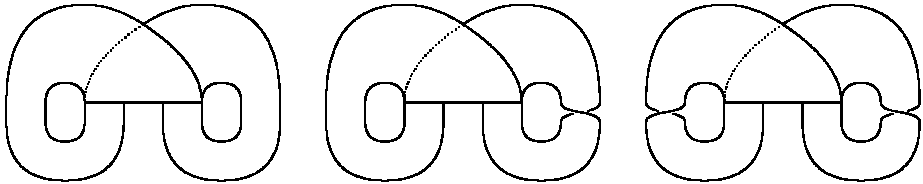
\begin{center}
\setlength{\unitlength}{0.26mm}
\begin{picture}(150,90)(-75,-40)

\qbezier(-70,0)(-70,50)(-30,50)
\qbezier(-30,50)(-15,50)(0,40)
\qbezier(0,40)(30,20)(30,0)

\qbezier(70,0)(70,50)(30,50)
\qbezier(30,50)(15,50)(0,40)
\bezier{30}(0,40)(-30,20)(-30,0)

\qbezier(-70,-10)(-70,-40)(-40,-40)
\qbezier(-40,-40)(-10,-40)(-10,-10)

\qbezier(70,-10)(70,-40)(40,-40)
\qbezier(40,-40)(10,-40)(10,-10)

\qbezier(-50,0)(-50,10)(-40,10)
\qbezier(-40,10)(-30,10)(-30, 0)
\qbezier(-50,-10)(-50,-20)(-40,-20)
\qbezier(-40,-20)(-30,-20)(-30, -10)

\qbezier(50,0)(50,10)(40,10)
\qbezier(40,10)(30,10)(30, 0)
\qbezier(50,-10)(50,-20)(40,-20)
\qbezier(40,-20)(30,-20)(30, -10)

\put(-70,0){\line(0,-1){10}}
\put(-50,0){\line(0,-1){10}}
\put(-30,0){\line(0,-1){10}}
\put(-10,0){\line(0,-1){10}}

\put(70,0){\line(0,-1){10}}
\put(50,0){\line(0,-1){10}}
\put(30,0){\line(0,-1){10}}
\put(10,0){\line(0,-1){10}}

\put(-30,0){\line(1,0){60}}

\end{picture}
\ \
\begin{picture}(150,90)(-75,-40)

\qbezier(-70,0)(-70,50)(-30,50)
\qbezier(-30,50)(-15,50)(0,40)
\qbezier(0,40)(30,20)(30,0)

\qbezier(70,0)(70,50)(30,50)
\qbezier(30,50)(15,50)(0,40)
\bezier{30}(0,40)(-30,20)(-30,0)

\qbezier(-70,-10)(-70,-40)(-40,-40)
\qbezier(-40,-40)(-10,-40)(-10,-10)

\qbezier(70,-10)(70,-40)(40,-40)
\qbezier(40,-40)(10,-40)(10,-10)

\qbezier(-50,0)(-50,10)(-40,10)
\qbezier(-40,10)(-30,10)(-30, 0)
\qbezier(-50,-10)(-50,-20)(-40,-20)
\qbezier(-40,-20)(-30,-20)(-30, -10)

\qbezier(50,0)(50,10)(40,10)
\qbezier(40,10)(30,10)(30, 0)
\qbezier(50,-10)(50,-20)(40,-20)
\qbezier(40,-20)(30,-20)(30, -10)

\put(-70,0){\line(0,-1){10}}
\put(-50,0){\line(0,-1){10}}
\put(-30,0){\line(0,-1){10}}
\put(-10,0){\line(0,-1){10}}

\put(30,0){\line(0,-1){10}}
\put(10,0){\line(0,-1){10}}

\qbezier(63.916, -4.2432)(70., -2.73)(70., 0.)
\qbezier(56.084, -5.7568)(50., -7.27)(50., -10.)
\qbezier(60., -5.)(50., -3.5)(50., 0.)
\qbezier(60., -5.)(70., -6.5)(70., -10.)

\put(-30,0){\line(1,0){60}}

\end{picture}
\ \
\begin{picture}(150,90)(-75,-40)

\qbezier(-70,0)(-70,50)(-30,50)
\qbezier(-30,50)(-15,50)(0,40)
\qbezier(0,40)(30,20)(30,0)

\qbezier(70,0)(70,50)(30,50)
\qbezier(30,50)(15,50)(0,40)
\bezier{30}(0,40)(-30,20)(-30,0)

\qbezier(-70,-10)(-70,-40)(-40,-40)
\qbezier(-40,-40)(-10,-40)(-10,-10)

\qbezier(70,-10)(70,-40)(40,-40)
\qbezier(40,-40)(10,-40)(10,-10)

\qbezier(-50,0)(-50,10)(-40,10)
\qbezier(-40,10)(-30,10)(-30, 0)
\qbezier(-50,-10)(-50,-20)(-40,-20)
\qbezier(-40,-20)(-30,-20)(-30, -10)

\qbezier(50,0)(50,10)(40,10)
\qbezier(40,10)(30,10)(30, 0)
\qbezier(50,-10)(50,-20)(40,-20)
\qbezier(40,-20)(30,-20)(30, -10)

\put(-30,0){\line(0,-1){10}}
\put(-10,0){\line(0,-1){10}}

\put(30,0){\line(0,-1){10}}
\put(10,0){\line(0,-1){10}}

\qbezier(-63.916, -4.2432)(-70., -2.73)(-70., 0.)
\qbezier(-56.084, -5.7568)(-50., -7.27)(-50., -10.)
\qbezier(-60., -5.)(-50., -3.5)(-50., 0.)
\qbezier(-60., -5.)(-70., -6.5)(-70., -10.)

\qbezier(63.916, -4.2432)(70., -2.73)(70., 0.)
\qbezier(56.084, -5.7568)(50., -7.27)(50., -10.)
\qbezier(60., -5.)(50., -3.5)(50., 0.)
\qbezier(60., -5.)(70., -6.5)(70., -10.)

\put(-30,0){\line(1,0){60}}

\end{picture}
\end{center}


Let $M$ be a connected compact smooth $3$-manifold with a smooth flow $\phi_t$, which has a hyperbolic chain-recurrent set $\mathcal{R}(\phi_t)$. By the Template Theorem of Birman and Williams \cite{BW2}, there is a template $\fT$ embedded in $M$ with a semi-flow $\varphi_t$ such that the periodic orbits of saddle of $\phi_t$ are in bijective correspondence with the periodic orbits of $\varphi_t$, and for any finitely many orbits the correspondence can be via ambient isotopy.

Here an ambient isotopy is a family of diffeomorphisms $\{h_{\lambda}\}_{\lambda \in I}$ of the ambient manifold $M$, depending smoothly on $\lambda$ and satisfying that $h_0 = 1_M$. We can define the ambient isotopy of embedding templates.


\begin{definition}
Two embedded templates $(\fT,\varphi_t)$ and $(\fT' ,\varphi'_t)$ in a $3$-manifold $M$ are said to be ambient isotopic if there is an ambient isotopy
$\{ h_\la \}_{\la \in I}$ such that
$h_1 (\fT)= \fT'$
and
$h_1\circ \varphi_t= \varphi_t '$.
\end{definition}


If two templates $(\fT,\varphi_t)$ and $(\fT' ,\varphi_t ')$ are ambient isotopic, then they are homeomorphic as branched surfaces and topologically equivalent as dynamical systems, as well as have the same embedding properties.


Next we consider the effect of an ambient isotopy acting on a template of a flow.


\begin{proposition} \label{isot}
Let $\phi_t$ be a flow on $M^3$ with hyperbolic chain-recurrent set. Suppose that $\{ h_\la \}_{\la \in I}$ is an ambient isotopy of $M$. Then the chain-recurrent set of the flow $h_1\circ \phi_t$ is hyperbolic. Moreover, if $\fT$ is a template of $\phi_t$, then $h_1 (\fT)$ is a template of $h_1\circ \phi_t$.
\end{proposition}


\begin{proof}
The differential of a diffeomorphism keeps the $\phi_t$-invariant decomposition of tangent bundle
$$\mathrm{T}M = \mathrm{E}^c\oplus \mathrm{E}^s\oplus \mathrm{E}^u.$$
Thus ambient isotopy preserves the chain-recurrent set decomposition clearly. It follows that $\{ h_\la \circ \phi_t\}_{\la \in I}$ is a family of flows on $M$, all of which have hyperbolic chain-recurrent sets. Particularly, so does $h_1\circ \phi_t$. Then from the constructing procedure of a template, following \cite{BW2, GHS}, $h_1 (\fT)$ is a template of $h_1\circ \phi_t$.
\end{proof}


Besides the ambient isotopy, there are two additional template moves adopted for defining equivalence among embedded templates, the splitting move and the switch move \cite{BW2, KSS}.
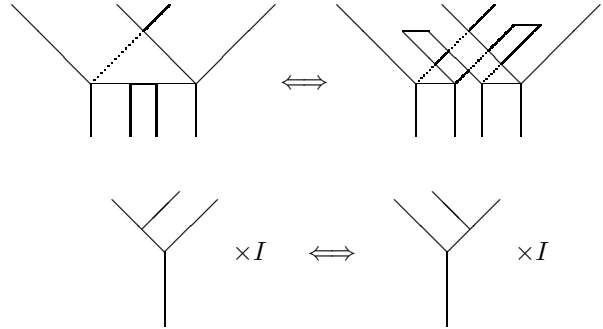
\begin{center}
\setlength{\unitlength}{0.35mm}
\begin{picture}(100,50)(-50,-20)
\put(-20,0){\line(0,-1){20}}
\put(-5,0){\line(0,-1){20}}
\put(20,0){\line(0,-1){20}}
\put(5,0){\line(0,-1){20}}

\bezier{15}(-20,0)(-10,10)(0,20)
\qbezier(10,30)(5,25)(0,20)

\put(-20,0){\line(-1,1){30}}
\put(20,0){\line(-1,1){30}}
\put(20,0){\line(1,1){30}}

\put(-5,0){\line(1,0){10}}
\linethickness{.1pt}
\put(-20,0){\line(1,0){15}}
\put(5,0){\line(1,0){15}}
\put(9,0){\makebox(0,0)[lc]{$\quad \quad \quad \quad \iff$}}
\end{picture}
\ \ \ \ \ \
\begin{picture}(100,50)(-50,-20)
\put(-20,0){\line(0,-1){20}}
\put(-5,0){\line(0,-1){20}}
\put(20,0){\line(0,-1){20}}
\put(5,0){\line(0,-1){20}}

\bezier{19}(-20,0)(-10,10)(0,20)
\qbezier(-12.5,7.5)(-10,10)(-7.5,12.5)
\qbezier(10,30)(5,25)(0,20)
\put(-20,0){\line(-1,1){30}}
\put(20,0){\line(-1,1){30}}
\put(20,0){\line(1,1){30}}

\put(-5,0){\line(-1,1){20}}
\put(5,0){\line(-1,1){20}}
\put(-25,20){\line(1,0){10}}
\qbezier(-5,0)(-2.5,2.5)(0,5)
\bezier{8}(0,5)(3.75,8.75)(7.5,12.5)
\qbezier(7.5,12.5)(12.5,17.5)(17.5,22.5)
\bezier{8}(5,0)(8.75,3.75)(12.5,7.5)
\qbezier(12.5,7.5)(20,15)(27.5,22.5)
\put(17.5,22.5){\line(1,0){10}}

\linethickness{.1pt}
\put(-20,0){\line(1,0){15}}
\put(5,0){\line(1,0){15}}
\end{picture}

\setlength{\unitlength}{1.0mm}
\begin{picture}(20,25)(-10,-10)
\put(0,0){\line(0,-1){10}}
\put(0,0){\line(1,1){7}}
\put(0,0){\line(-1,1){7}}
\put(-3,3){\line(1,1){5}}
\put(9,0){\makebox(0,0)[lc]{$\times I \quad \iff$}}
\end{picture}
\ \ \ \ \ \ \ \ \ \ \ \ \ \
\begin{picture}(20,25)(-10,-10)
\put(0,0){\line(0,-1){10}}
\put(0,0){\line(1,1){7}}
\put(0,0){\line(-1,1){7}}
\put(3,3){\line(-1,1){5}}
\put(9,0){\makebox(0,0)[lc]{$\times I$}}
\end{picture}

\end{center}
A dynamical reason for introducing these two moves is from the Conley index theory \cite{C}, that the dynamics of a invariant set of a flow is invariant if the relative homotopy type of the index pair of this invariant set is unchanged \cite{Sa}.

However, obviously the splitting move changes the homeomorphism type of templates, and still less the isotopy type. This is the essential difficulty to define an isotopic invariant for embedded templates, such that equivalent templates have the same invariant.

In the next section, we remedy it by choosing a fixed homeomorphism type of an embedded template, and then merely consider the action of ambient isotopy on the template without involving the two template moves.


\section{Templates and spatial graphs}

This section focuses on the boundary of templates. If the ambient manifold is a $3$-space, \ie $\R^3$ or $S^3$, the boundary of a template can be regarded as a trivalent spatial graph. We have the following observation.


\begin{proposition} \label{isog}
Let $\fT$ be a template in $\R^3$ or $S^3$, and $\{ h_\la \}_{\la \in I}$ be an ambient isotopy. Then $\pa \fT$ is isotopic to $\pa h_1 (\fT)$ as spatial graph.
\end{proposition}


\begin{proof}
The proof is straightforward by
$$\pa h_1(\fT)= h_1 (\pa \fT)\sim \pa \fT,$$
where the equality holds because of the diffeomorphism $h_1$ commuting with $\pa$, and $\sim$ means isotopy of spatial graphs due to the ambient isotopy.
\end{proof}


From the viewpoint of topology, the spatial graph is a bit more complex object than the knot or link. Now we use Kauffman's invariant of trivalent spatial graphs to dispose the boundary of a template, and get a collection of knots and links as our isotopic invariant of embedded templates.

Let us recall the definition of Kauffman's invariant in \cite{K}. For a trivalent spatial graph $G$ in a $3$-space, there are three types of local replacement at a vertex $v$ of $G$ as follows.
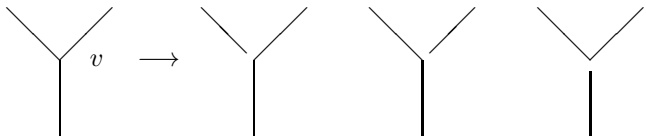
\begin{center}
\setlength{\unitlength}{1.0mm}
\begin{picture}(20,28)(-10,-14)
\put(0,0){\line(0,-1){10}}
\put(0,0){\line(1,1){7}}
\put(0,0){\line(-1,1){7}}
\put(4,0){\makebox(0,0)[lc]{$v \quad \longrightarrow$}}
\end{picture}
\ \ \ \
\begin{picture}(20,28)(-10,-14)
\put(0,0){\line(0,-1){10}}
\put(0,0){\line(1,1){7}}
\put(-1,1){\line(-1,1){6}}
\end{picture}
 \
\begin{picture}(20,28)(-10,-14)
\put(0,0){\line(0,-1){10}}
\put(1,1){\line(1,1){6}}
\put(0,0){\line(-1,1){7}}
\end{picture}
 \
\begin{picture}(20,28)(-10,-14)
\put(0,-1.5){\line(0,-1){8.5}}
\put(0,0){\line(1,1){7}}
\put(0,0){\line(-1,1){7}}
\end{picture}
\end{center}
Let $\tau(G)$ be the collection of knots and links obtained by making a local replacement of all the three types at each vertex of $G$.
Moreover, for a path or cycle $\alpha$ in $G$, we define the generalized Kauffman's invariant of $G$ by
$$\tau_\alpha (G)\defeq \{l\in \tau(G):\alpha \subset l\}.$$
That is, the subset of knots and links in $\tau(G)$ which contain $\alpha$.

For example, for a unknotted $\theta$-curve $G$ with three edges $a,b$ and $c$, $\tau(G)$ is the set of three trivial knots $\{ab,bc,ac \}$. Hence, the generalized Kauffman's invariant $\tau_b(G)$ is the set $\{ab,bc \}$ consisting of two unknots.

We say that two Kauffman's invariants $\tau(G)$ and $\tau(G')$ are equal if they are the same set of link types. Similar to generalized Kauffman's invariants.

\begin{theorem}(Kauffman \cite{K}) \label{tau}
For a trivalent spatial graph $G$ in a $3$-space, $\tau(G)$, taken up to ambient isotopy, is a topological invariant of $G$.
\end{theorem}


Here "topological" means embedding.


\begin{corollary} \label{taua}
$\tau_\alpha (G)$ is an isotopic invariant of $G$ for any path or cycle $\alpha$ in $G$.
%
\end{corollary}


Combining Proposition~\ref{isog} and Kauffman's theorem~\ref{tau}, we get the following consequence immediately.


\begin{theorem} \label{taut}
Let $\fT$ ba a template in a $3$-space, and $\{ h_\la \}_{\la \in I}$ be an ambient isotopy. Then the Kauffman's invariants of $\pa \fT$ and $\pa h_1 (\fT)$ are equal, i.e.
$\tau (\pa \fT) = \tau (\pa h_1(\fT))$
as sets of link types.
\end{theorem}


\begin{remark}
The effect of the splitting and the switch moves on the Kauffman's invariant of a template boundary is explicit. For an embedded template $\fT$, the set $\tau (\pa \fT)$ of knots and links is invariant under the switch move. On the other hand, from the template diagram \cite{KSS} we can see that doing a splitting move induces adding or decreasing exactly one link to $\tau (\pa \fT)$, which is split into a link in the original set with a unlinked trivial component.
\end{remark}


In Section 5, we will apply this invariant to study the isotopic classification of simple Smale flows on $S^3$.


\section{A complete invariant of templates}

Proposition~\ref{isog} states that ambient isotopic templates have isotopic boundaries. Conversely, we have the following criterion on determining if two templates in $3$-sphere are ambient isotopic. Therefore, it is equivalent to say that the template boundary, as a spatial graph, is a complete invariant of embedded templates with a fixed homeomorphism type.


\begin{theorem} \label{crit}
Let $\fT$ and $\fT'$ be two embedded templates in $S^3$, which are topologically equivalent as dynamical systems. If they have isotopic boundaries
$\pa \fT\sim \pa \fT'$,
then $\fT$ and $\fT'$ are ambient isotopic templates.
\end{theorem}


The construction of ambient isotopies relies on the Isotopy Extension Theorem \cite{Hir}, which states that an isotopy of a compact submanifold of $S^3$ always extends to an ambient isotopy having compact support. To prove Theorem~\ref{crit}, we need the following two lemmas on embedded discs, annuli, and M\"{o}bius strips.


\begin{lemma}\label{lemd}
If two embedded discs $D$ and $D'$ in $S^3$ have the same boundaries $\pa D=\pa D'$, then there is an ambient isotopy sending $D$ to $D'$.
\end{lemma}


\begin{proof}
By transversality, $D$ can be perturbed by a small isotopy to intersect $D'$ transversally. The intersection is the union of a submanifold of dimension $1$ and the common boundary, then it consists of three parts: $\pa D=\pa D'$, several families of finite concentric circles, and a finite number of intervals whose endpoints lie on the boundary.

First, consider the circles in $D\cap D'$. Choose a family of concentric circles in $D$ with innermost one $C$, which is also a innermost one in $D'$. The two discs $D_1$ and $D_1'$, bounded by $C$ in $D$ and $D'$ respectively, form an embedded $2$-sphere in $S^3$. By Alexander's theorem on sphere \cite{Ha}, the $2$-sphere $D_1\cup D_1'$ bounds an embedded $3$-ball $B$ in $S^3$. Then by an isotopy of $D$ supported near $C$ we can push $D$ across $B$, eliminating $C$ and decreasing by $1$ the number of components of $D\cap D'$. Repeat this step of eliminating circle components of $D\cap D'$ as long as only $\pa D=\pa D'$ remains.

Second, consider the intervals in $D\cap D'$. We can find out an "outmost" interval $\alpha$ in $D\cap D'$. Hence, there is a sub-arc $\beta$ in $\pa D$ such that $\alpha \cup \beta$ bounds two discs $D_2$ and $D'_2$ in $D$ and $D'$ respectively. $D_2$ and $D'_2$ intersect exactly at the circle $\alpha \cup \beta$. As for $D_1$ and $D_1'$, an isotopy of $D$, supported near $\alpha \cup \beta$, eliminates $\alpha$ and decreases by $1$ the number of intervals in $D\cap D'$. Thus doing this step repeatedly allows us to eliminate all the intersection intervals.

Now $D$ and $D'$ intersect only at $\pa D=\pa D'$ and form an embedded $2$-sphere in $S^3$, which bounds an embedded $3$-ball giving an isotopy of $D$ and $D'$.

According to the Isotopy Extension Theorem, all the above isotopies extend to ambient isotopies of $S^3$. Gluing these finite number of isotopies gives an ambient isotopy sending $D$ to $D'$.
\end{proof}


\begin{lemma}\label{lema}
Let $A$ and $A'$ be two embedded annuli or two embedded M\"{o}bius strips in $S^3$. If $\pa A=\pa A'$, then there is an ambient isotopy sending $A$ to $A'$.
\end{lemma}


\begin{proof}
Let $A$ and $A'$ be two embedded annuli in $S^3$. First consider the special case that $A=A_0$ is a standard annulus, \ie $A_0$ is unknotted and its two boundary circles are unlinked.

By transversality, after a small isotopy $A_0$ and $A'$ intersect transversally. The intersection is composed of $\pa A_0 =\pa A$, a finite number of circles, and a finite number of intervals with endpoints lying on the boundary. We can repeat the procedure in the proof of the previous lemma to eliminate all the inessential circles in $A_0\cap A'$.

For an interval $\alpha$ in $A_0\cap A'$, if it connects the two components of $\pa A_0$, then $A_0\cap A'$ contains no essential circles. $\alpha$ cuts $A_0$ and $A$ into discs, and by Lemma~\ref{lemd} $A_0$ and $A$ are isotopic. If the two endpoints of $\alpha$ lie on one of the two boundary circles, there is a sub-arc $\beta$ in this component satisfying that $\alpha \cup \beta$ bounds two discs $D$ and $D'$ in $A_0$ and $A'$ respectively, and
$$D\cap D' =\alpha \cup \beta.$$
Then as in Lemma~\ref{lemd}, an isotopy of $A_0$, supported near $\alpha \cup \beta$, eliminates $\alpha$ and decreases by $1$ the number of intervals in $A_0\cap A'$. Thus by repeating this step we can eliminate all the intersection intervals.

For an essential circle $C$ in $A_0\cap A'$, if there is another essential one $C'$ adjacent to $C$ and not in the boundary, $C$ together with $C'$ bounds two annuli $A_1$ and $A_1'$ in $A_0$ and $A'$ respectively. $A_1$ and $A_1'$ intersect transversally only at boundaries thus form an embedded torus in $S^3$. By Alexander's theorem on torus \cite{Ha}, the torus $A_1\cup A_1'$ bounds a solid torus $X$ in $S^3$. Then by an isotopy of $A_0$ supported near $C\cup C'$, we can push $A_0$ across $X$, eliminating $C$ and $C'$, and decreasing by $2$ the number of circle components of $A_0\cap A'$. Repeat this step of eliminating circles in $A_0\cap A'$, as long as only $\pa A_0=\pa A'$ remains, or with a single essential circle.

If no other circle than $\pa A_0=\pa A'$ remains, $A_0$ and $A'$ form an embedded torus in $S^3$, which bounds an embedded solid torus giving an isotopy of $A_0$ and $A'$. If there is exactly one essential circle $C''$ in $A_0\cap A'$ but not in boundary, $C''$ together with a component of $\pa A_0=\pa A'$ bounds two annuli in the original annuli respectively. As above, they bound an embedded solid torus giving an isotopy of $A_0$ and $A'$ to eliminate $C''$, and then no other circles than $\pa A_0=\pa A'$ exists.

According to the Isotopy Extension Theorem, all the above isotopies extend to ambient isotopies of $S^3$. Gluing these finite number of isotopies gives an ambient isotopy sending $A_0$ to $A'$.

For general annulus $A$, there is a homeomorphism $h\col S^3\to S^3$ satisfying $h(A) =A_0$. By the special case, there is an isotopy $h_\la \col S^3\to S^3$ such that
$$h_0= 1_{S^3},\quad h_0 (A_0)= A_0,\quad h_1 (A_0)= h(A').$$
Then $h^{-1}\circ h_\la \circ h$ is an ambient isotopy sending $A$ to $A'$.

The proof for the case of M\"{o}bius strips is similar, and is simpler than the case of annuli by the fact that if there is an interval in $A\cap A'$ with endpoints lying on the boundary, then the intersection contains no essential circles.
\end{proof}


We can now prove Theorem~\ref{crit}.


\begin{proof}
$\fT$ and $\fT'$ are topologically equivalent templates thus are homeomorphic as branched surfaces. We construct an ambient isotopy of $S^3$ sending $\fT$ to $\fT'$ in three steps. Then the semi-flow condition holds clearly, and the two templates are ambient isotopic.

Assume that $\fT$ and $\fT'$ are both connected. Otherwise, for each component we construct a desired isotopy supported near it, and glue them together.

(1) Since $\fT$ and $\fT'$ have isotopic boundaries, after a suitable isotopy of $S^3$ they have common boundary
$\pa \fT =\pa \fT'$.

(2) Let us tackle branch lines. Let $\alpha_1,\cdots, \alpha_n$ be the branch lines of $\fT$ while $\alpha'_1,\cdots, \alpha'_n$ be that of $\fT'$, satisfying that $\alpha_i$ and $\alpha'_i$ have the same endpoints on
$\pa \fT =\pa \fT'$
for $i=1,\cdots, n$. According to the template diagrams of $\fT$ and $\fT'$ \cite{KSS}, for every $i$, the circle consisting of $\alpha_i$ and $\alpha'_i$ cannot be knotted. As in the proof of Lemma~\ref{lemd} and by the Schoenflies Theorem, $\alpha_i$ and $\alpha'_i$ form the boundary circle of an embedded disc in $S^3$, which gives an isotopy of the two arcs relative the endpoints and with support near the circle. Extending these $n$ isotopies and gluing them together give an isotopy of $S^3$ sending each $\alpha_i$ to $\alpha'_i$ and preserving
$\pa \fT= \pa \fT'$.

(3) The branch lines divide the template into pieces which can be viewed as splitting charts. For a piece $S$ of $\fT$, let $S'$ be the corresponding one in $\fT'$, then one of the following three cases occurs.

(i) If the two bands of $S$ are disjoint with its branch line, $S$ and $S'$ are embedded discs in $S^3$. By Lemma~\ref{lemd}, there is an ambient isotopy sending $S$ to $S'$ and preserving
\begin{equation*}
\pa \fT =\pa \fT', \quad \alpha_i=\alpha'_i, i=1,\cdots, n.
\end{equation*}

(ii) If one band is disjoint with the branch line while the other one is not, both $S$ and $S'$ have a closed band, and are either embedded annuli or M\"{o}bius strips. Then by Lemma~\ref{lema}, an ambient isotopy sends $S$ to $S'$ and preserves the coincidence of the boundary and branch lines.

(iii) If all the two bands are closed, by connectivity both $S$ and $S'$ have exactly one branch line thus are Lorenz-like templates.
The exit line $\beta$ of $S =\fT$ sweeps a square domain $B$ by going backward through the semi-flow until touching the branch line $\alpha$. So does for
$S' =\fT'$ to get a square domain $B'$. Then $B$ and $B'$ have a pair of opposite common sides, the exit line $\beta =\beta'$ and an interval $\gamma =\gamma'$ of the branch line, while the other sides, $\delta_1, \delta_2$ in $\fT$ and $\delta'_1, \delta'_2$ in $\fT'$, are part of flow lines with the same endpoints
$$\pa \delta_1=\pa \delta'_1, \quad \pa \delta_2=\pa \delta'_2.$$
As in Step (2),
there is an ambient isotopy sending each $\delta_i$ to $\delta'_i$ and preserving
$$\pa \fT= \pa \fT', \quad \alpha =\alpha'.$$
Now the two discs $B$ and $B'$ have common boundary, Lemma~\ref{lemd} implies they are ambient isotopic. Both the complements $\fT -B$ and $\fT' -B'$ consist of two bands, we can cut them off along branch lines and use Lemma~\ref{lema} twice to isotope $\fT -B$ to $\fT' -B'$.

To sum up, there is always an ambient isotopy $h_t$ sending $S$ to $S'$ and preserving $\pa \fT =\pa \fT'$. Moreover, we can ask it to be supported near $S$ by the compactness.

For another piece of $\fT$ an ambient isotopy, sending it to the corresponding one of $\fT'$, can be constructed as above. It can be required to preserve the coincidence of the boundary and branch lines, as well as
$$h_1(S) =S'.$$
Thus gluing $n$ ambient isotopies for all the $n$ pieces of $\fT$ gives an isotopy of $S^3$ sending $\fT$ to $\fT'$. Now we finish Step (3).

All the ambient isotopies constructed in the three steps yield an ambient isotopy sending $\fT$ to $\fT'$. This completes the proof.
\end{proof}


The proof of the previous two lemmas works as well as for ambient manifolds with trivial $\pi_2$. This implies that the conclusion of Theorem~\ref{crit} holds in those cases. We have the following corollary.


\begin{corollary}
Let $M$ be an irreducible closed $3$-manifold. Let $\fT$ and $\fT'$ be two embedded templates in $M$, which are topologically equivalent as dynamical systems. If they have isotopic boundaries
$\pa \fT\sim \pa \fT'$,
then $\fT$ and $\fT'$ are ambient isotopic as templates.
\end{corollary}


\begin{remark}
Theorem~\ref{crit} can be viewed as an isotopic version of Conley index theory. For a flow $\phi_t$ on $S^3$, a saddle set $\Om$ induces a template $\fT$. Then the thicken template $\bar{\fT}$ is an isolated neighborhood of $\Om$, while the boundary $\pa \fT$ is exactly the core of the exit set $\pa_- \bar {\fT}$. Our theorem shows that the embedding of exit set $\pa_- \bar {\fT}$ totally determines not only the homeomorphism type of the Conley index pair $(\bar{\fT}, \pa_- \bar {\fT})$, but also the isotopy type of their embedding in the underlying $3$-manifold.
\end{remark}


\section{Isotopy of simple Smale flows}
\label{secMldCalc}

In dynamics, topologically equivalent flows have the same orbit structure hence the same dynamical properties, but their corresponding orbits may be embedded in the underlying manifold in different ways, or in other words by physicists, they may have different orbit organizations \cite{Gil}. Embedding of objects in the ambient space is usually under consideration in topology up to isotopy \cite{Hir}.

In this section, we define the isotopy of flows, which is more subtle than the topological equivalence in the viewpoint of classifying. It is naturally related to the ambient isotopy and is useful to describe the embedding of orbits into the underlying manifold. Then we concentrate on the classification of simple Smale flows on $S^3$ up to isotopy by using the induced spatial graph and its generalized Kauffman's invariant.

Let $M$ be a connected compact smooth manifold. All the flows we consider here are supposed to be smooth.


\begin{definition}
Two flows $\phi_t$ and $\psi_t$ on $M$ are isotopic if one of them can be deformed to the other through smooth flows, \ie there is a smooth map
$$H\col \R \times M\times I \to M, \quad (t, x, \la)\mapsto H_\la (t, x)$$
such that $H_0 (t, x) =\phi_t (x)$, $H_1 (t, x) =\psi_t (x)$, and for any $\la \in I$, $H_\la (\cdot, \cdot)$ is a smooth flow on $M$. Then we denote
$\phi_t \sim \psi_t$, and call $H_\la$ an isotopy from $\phi_t$ to $\psi_t$.
\end{definition}


The isotopy of two flows is an equivalent relation (for the transitivity, see p. 111 of \cite{Hir}). There is a close relationship between isotopy of flows and ambient isotopy. If $\phi_t$ is a flow on $M$, $\{ h_\la \}_{\la \in I}$ is an ambient isotopy, then $\phi_t = h_0 \circ \phi_t$ and $h_1 \circ \phi_t$ are isotopic by defining
$$H_\la (t, x)= h_\la \circ \phi_t(x), \quad (t, x)\in \R \times M$$
for $\la \in I$. Conversely, we have the following proposition.


\begin{proposition}
Let $\phi_t$ and $\psi_t$ be two flows on $M$ which are isotopic. Then there exists an ambient isotopy $\{ h_\la \}_{\la \in I}$ so that
$h_1\circ \phi_t =\psi_t$.
\end{proposition}


\begin{proof}
If $H_\la$ is an isotopy from $\phi_t$ to $\psi_t$, then
$$h_\la (y)= H_\la (t, \phi_{-t} (y)), \quad y\in M$$
gives an ambient isotopy for $\la \in I$, satisfying that $h_0\circ \phi_t =\phi_t$ and $h_1\circ \phi_t =\psi_t$.
\end{proof}


Next we consider the relationship between topological equivalence and isotopy of flows in dynamics. For a flow $\phi_t\col M\to M$, denote the orbit of $x\in M$ by
\begin{equation*}
O_{\phi} (x)=\{ \phi_t (x): t\in \R\}.
\end{equation*}


\begin{proposition}
If two flows are isotopic, then they are topologically equivalent.
\end{proposition}


\begin{proof}
Suppose that $H_\la$ is an isotopy from $\phi_t$ to $\psi_t$ on $M$. As above,
$$h_\la (x)= H_\la (t, \phi_{-t} (x)), \quad x\in M$$
gives an ambient isotopy for $\la \in I$. Then $h_1$ is a diffeomorphism which takes orbits of $\phi_t$ to orbits of $\psi_t$ and preserves the time orientation, \ie for any $x\in M$,
$$h_1 (O_\phi (x))=O_\psi (h_1 (x)),$$
and for any $x\in M$, $\eps >0$, there is a $\delta > 0$ such that $\forall t\in (0,\delta)$, $\exists s\in (0,\eps)$,
$$h_1 (\phi_t (x))=\psi_s (h_1 (x)).$$
Thus $\phi_t$ and $\psi_t$ are topologically equivalent by $h_1$.
\end{proof}


Now we focus on simple Smale flows. First we give the general definition of the Smale flow.


\begin{definition} \cite{Fr}
Let $M$ be a connected orientable closed $3$-manifold. A flow $\phi_t\col M\to M$ is called a Smale flow if (1) its chain-recurrent set $\mathcal{R}(\phi_t)$ is hyperbolic, (2) the basic sets are of dimensions $0$ or $1$, and (3) for any $x, y\in \mathcal{R}(\phi_t)$, the stable manifold $W^s (x)$ of $x$ and the unstable manifold $W^u (y)$ of $y$ have transversal intersection.

In particular, a simple Smale flow (SSF) is a Smale flow whose chain-recurrent set $\mathcal{R}(\phi_t)$ is composed of exactly three basic sets- a closed orbit attractor $a$, a closed orbit repeller $r$, and a nontrivial saddle set $\Om$ \cite{Su, Y}.
\end{definition}


Next we consider SSFs on $3$-sphere by using spatial graphs they induces and the generalized Kauffman's invariant.


\begin{theorem} \label{lemg}
Let $\phi_t$ be a simple Smale flow on $S^3$ with attractor $a$ and a template $\fT$. Then the spatial graph composed of $a$ and the boundary of $\fT$
$$G\defeq a\cup \pa \fT$$
is invariant under isotopy of $\phi_t$. Furthermore, the generalized Kauffman's invariant $\tau_a (G)$ is an isotopic invariant of $\phi_t$.
\end{theorem}


\begin{proof}
Let
$\{ h_\la \}_{\la \in I}$
be an ambient isotopy. Due to isotopy preserving the chain-recurrent set decomposition all the time and $h_0\circ \phi_t =\phi_t$,
$\{ h_\la \circ \phi_t\}_{\la \in I}$
is a family of SSFs. Therefore $h_1(a)$ is the closed orbit attractor of $h_1\circ \phi_t$, and it is isotopic to $a$. By Proposition~\ref{isog}, $\pa \fT$ is isotopic to $\pa h_1(\fT)$. Since all the linking relationships are preserving under ambient isotopy,
$$h_1(G) = h_1(a)\cup h_1(\pa \fT) = h_1(a)\cup \pa h_1(\fT) \sim a\cup \pa \fT = G$$
as spatial graphs.

By Corollary~\ref{taua}, $\tau_a (G)$ is also invariant under $\{ h_\la \}_{\la \in I}$.
\end{proof}


As for Proposition~\ref{isog}, we can ask the inverse questions of Theorem~\ref{lemg} that

(1) wether or not the spatial graph $G= a\cup \pa \fT$ is a complete isotopic invariant of SSFs on $S^3$; and

(2) wether or not the generalized Kauffman's invariant $\tau_a (G)$ is complete.


For the first question, we give a general discussion on the structure of SSFs \cite{Su, Y}, which is fundamental for the classification problem. Let $\phi_t$ be a SSF on a $3$-manifold $M$ with attractor $a$, repeller $r$, and saddle set $\Om$.

For the saddle set $\Om$, choose a template $\fT$. As a topological space, $\fT$ is a connected compact branched $2$-manifold. The thicken template $\bar{\fT}$ is an isolated neighborhood of $\Om$, which is generally a knotted handlebody embedded in $M$. Assume the genus of $\bar{\fT}$ is $g$, then $\pa \bar{\fT}$ is a connected orientable closed surface of genus $g$. It is composed of two compact subsurfaces- the exit set $\pa_- \bar{\fT}$ and the entrance set
$\pa_+ \bar{\fT}$, which are in fact homeomorphic
$$\pa_- \bar{\fT} \cong \pa_+ \bar{\fT}$$
and intersect at a finite number of disjoint circles as their common boundary. That is,
$$\pa (\pa_- \bar{\fT})= \pa (\pa_+ \bar{\fT})$$
is a disjoint union of finite circles. The core of the exit set $\pa_- \bar{\fT}$ is exactly $\pa \fT$, the boundary of the template.

Let $A$ and $R$ be the tubular neighborhoods of the attractor $a$ and the repeller $r$ in the underlying $3$-manifold $M$ respectively. Since $a$ and $r$ are generally knotted, $A$ and $R$ are knotted solid tori in $M$ with knotted tori $\pa A$ and $\pa R$ as their boundaries.

The original flow $(M, \phi_t)$ can be reconstructed by attaching $A, \bar{\fT}$ and $R$ along their boundaries in the way that gluing $\pa_- \bar{\fT}$ to $\pa A$, $\pa_+ \bar{\fT}$ to $\pa R$, and the remainder part of $\pa A$ to the remainder part of $\pa R$. Both the manifold and flow obtained from the attaching can be smooth by a modification such as in \cite{M}. Let
$$N\defeq A\cup \bar{\fT}$$
be the result of attaching $\bar{\fT}$ to $A$.


The following theorem gives a partial answer of the first question in the case of SSFs on $3$-sphere. It essentially relies on Theorem~\ref{crit} and the fact that the knot complement in $S^3$ is a complete invariant.


\begin{theorem} \label{flowg}
Let $\phi_t$ and $\psi_t$ be two simple Smale flows on $S^3$ with attractors $a$ and $a'$ respectively. Suppose that $\phi_t$ and $\psi_t$ restricted on their saddle sets are topologically equivalent. Then they have topologically equivalent templates, denoted by $\fT$ and $\fT'$ respectively. Moreover, if
$G= a\cup \pa \fT$ and $G'=a' \cup \pa \fT'$ are isotopic spatial graphs, then $\phi_t$ and $\psi_t$ are isotopic flows, after reversing the directions of the attractor or the repeller if necessary.
\end{theorem}


\begin{proof}
Since $\phi_t$ and $\psi_t$ restricted on their saddle sets are topologically equivalent, from the constructing procedure of templates following \cite{BW2, GHS}, they have topologically equivalent templates clearly.

Let $\fT$ and $\fT'$ be two topologically equivalent templates of $\phi_t$ and $\psi_t$ respectively. By the above reconstruction of a SSF, it is sufficient to construct an ambient isotopy preserving the isolated neighborhoods decomposition
$$S^3= A\cup \bar{\fT}\cup R$$
of $\phi_t$ for all time and connecting the two flows.

If $G= a\cup \pa \fT$ and $G'=a' \cup \pa \fT'$ are isotopic spatial graphs, there is an ambient isotopy $\{ f_\la \}_{\la \in I}$ sending $G$ to $G'$. Then for $f_1(\fT)$ and $\fT'$, $\pa f_1(\fT) =\pa \fT'$, one can construct an isotopy $\{ g_\mu \}_{\mu \in I}$ of $S^3$ sending $f_1(\fT)$ to $\fT'$ and keeping
$\pa f_1(\fT) =\pa \fT'$
as in the proof of Theorem~\ref{crit}. The isotopy $\{ g_\mu \}_{\mu \in I}$ has compact support near $f_1(\fT)$ thus can be required to fix $f_1(a) =a'$. Gluing $\{ f_\la \}_{\la \in I}$ and $\{ g_\mu \}_{\mu \in I}$ together yields an ambient isotopy $\{ h_\nu \}_{\nu \in I}$ satisfying that
$$h_1(G) =G', \quad h_1(\fT) =\fT'.$$
Then $\{ h_\nu \}_{\nu \in I}$ sends the thicken template $\bar{\fT}$ into $\bar{\fT'}$, as well as the tubular neighborhood $A$ into a tubular neighborhood $A'$, and preserves all the linking relationships including for bands of the template and for the attractor with the template. Therefore, the complements of the repeller $r$ and $r'$ are corresponding under $\{ h_\nu \}_{\nu \in I}$. Since the complement is a complete invariant of knots \cite{GL}, the isolated neighborhoods decomposition is invariant by $\{ h_\nu \}_{\nu \in I}$.

The argument leaves out the direction of the closed orbit attractor and repeller, \ie they are under consideration only as point-sets. Nevertheless, if necessary, we can reverse the direction of $a$ or $r$ such that $h_1\circ \phi_t =\psi_t$.
\end{proof}


For the second question that wether or not the generalized Kauffman's invariant $\tau_a (G)$ is a complete invariant for SSFs on $S^3$, the answer seems to be no. A reason is that the Kauffman's invariant is not powerful enough to distinguish spatial graphs. In fact, even for $\theta$-curves, non-isotopic spatial graphs may have the same Kauffman's invariant \cite{Zh}. Thus it is quite a question.


\begin{question}
For a simple Smale flow on $S^3$ with attractor $a$ and a template $\fT$, let $G= a\cup \pa \fT$. Suppose that the homeomorphism type of $\fT$ is fixed. What is the condition of $\fT$ such that the generalized Kauffman's invariant $\tau_a (G)$ is a complete isotopic invariant of the flow?
\end{question}


\begin{remark}
The invariant $\tau_a (G)$ is complete for SSFs on $S^3$ with Lorenz-like templates \cite{Su, Y} or the universal template and its analogues \cite{HS, AS}, up to the disc sign.
\end{remark}




\end{document}